\documentclass{amsart}
\usepackage{amsfonts}
\usepackage{color}
\usepackage{soul}
\usepackage[a4paper,top=3cm, bottom=3cm, left=3cm, right=3cm]{geometry}
\begin{document}

\vfuzz2pt 
\newcommand{\red}{\color{red}}
\newcommand{\blue}{\color{blue}}
\newcommand{\black}{\color{black}}

 \newtheorem{thm}{Theorem}[section]
 \newtheorem{cor}{Corollary}[section]
 \newtheorem{lem}{Lemma}[section]
 \newtheorem{prop}{Proposition}[section]
 \theoremstyle{definition}
 \newtheorem{defn}{Definition}[section]
 \theoremstyle{remark}
 \newtheorem{rem}{Remark}[section]
 \numberwithin{equation}{section}
\newcommand{\CC}{\mathbb{C}}
\newcommand{\KK}{\mathbb{K}}
\newcommand{\ZZ}{\mathbb{Z}}
\newcommand{\RR}{\mathbb{R}}
\def\a{{\alpha}}

\def\b{{\beta}}

\def\d{{\delta}}

\def\g{{\gamma}}

\def\l{{\lambda}}

\def\gg{{\mathfrak g}}
\def\cal{\mathcal }

\title{On Solvable Lie and Leibniz Superalgebras with maximal codimension of  nilradical}

\author{L.M. Camacho}
\address{Luisa Mar\'{i}a Camacho \newline \indent
Dpto. Matem{\'a}tica Aplicada I, Universidad de Sevilla, Sevilla
 (Spain) }
\email{lcamacho@us.es}

\author{R.M. Navarro}
\address{Rosa Mar{\'\i}a Navarro \newline \indent
Dpto. de Matem{\'a}ticas, Universidad de Extremadura, C{\'a}ceres
 (Spain) }
\email{rnavarro@unex.es}
\thanks{This work has been  supported  by Agencia Estatal de Investigaci\'on (Spain), grant MTM2016-79661-P (European FEDER support included, EU)}

\author{B.A. Omirov}
\address{ Bakhrom Omirov \newline \indent
National University of Uzbekistan, Tashkent (Uzbekistan)}
\email{{\tt omirovb@mail.ru}}

\begin{abstract} Along this paper we show that under certain conditions the method for describing of solvable  Lie and Leibniz algebras with maximal codimension of nilradical is also extensible to Lie and Leibniz superalgebras, respectively. In particular, we totally determine the solvable Lie and Leibniz superalgebras with maximal codimension of model filiform and model nilpotent nilradicals. Finally, it is established that the superderivations of the obtained superalgebras are inner.

\end{abstract}

\maketitle
{\bf 2010 MSC:} {\it  17A32; 17A70; 17B30;  17B40.}

{\bf Key-Words:} {\it  solvable Lie  superalgebras, solvable Leibniz superalgebras, derivations, nilpotent Lie superalgebras, nilpotent Leibniz superalgebras.}

\section{Introduction}

In $1950$ A.I. Malcev proved that a solvable Lie algebra is uniquely determined by its nilradical \cite{Malcev}. A decade later Mubarakzjanov proposed the method of description of solvable Lie algebra in terms of nilradical and its nil-independent derivations \cite{metodoLie}. He noted that the codimension of nilradical does not exceed the number of nil-independent derivations and the number of generators of the nilradical. Since then the classification of solvable Lie algebras with the Abelian, Heisenberg, filiform, quasi-filiform nilradicals are obtained, see for instance \cite{Casas1,Ancochea,Casas11,Casas13,Casas15, LAA2010}. Recently, the topic of research has drawn a lot of attention and in particular in \cite{metodoLeib} the authors extended Mubarakzjanov's method to Leibniz algebras. Therefore, along the last few years have been obtained classifications with diferent types of Leibniz nilradicals \cite{lisa,LMA, 2,2c}. Among all the solvable Lie and Leibniz algebras special mention deserve the maximal (in dimensional sense) solvable Lie and Leibniz algebras with given nilradical, due to its properties, they are in some cases cohomologically rigid \cite{rigid_algebras}.

The main goal of this paper is the study of maximal solvable Lie and Leibniz superalgebras with a given nilradical. It should be noted that the structures of solvable Lie and Leibniz superalgebras are more complex than structures of solvable Lie and Leibniz algebras \cite{Onindecomposable}. In particular, Lie's theorem is not true (neither in its general or its reduced forms) for a solvable Leibniz (respectively, Lie)  superalgebras $L$. Moreover, the square of a solvable superalgebra is not necessary to be nilpotent  \cite{codimension}. Nevertheless, throughout this paper we show that under certain conditions Mubarakzjanov's method of description of maximal solvable Lie and Leibniz algebras with given nilradical is also applicable for Lie and Leibniz superalgebras, respectively. In particular, we determine the maximal dimensional solvable Lie and Leibniz superalgebras with model filiform and model nilpotent nilradicals. Additionally, we compute the space of superderivations on the maximal-dimensional solvable Lie and Leibniz superalgebras with filiform and model nilpotent nilradical and show that all of these superderivations are inner. These results constitute an extension of the results obtained for similar Lie and Leibniz algebras. Finally, let us note that all Lie and Leibniz superalgebras obtained in this paper are an excellent candidates to cohomologycally rigid  superalgebras, we leave this study for a further work  though.

\section{Preliminary Results}

A vector space $V$ is said to be $\ZZ_2${\em -graded} if it admits a decomposition in direct sum, $V=V_{\bar 0} \oplus V_{\bar 1}$, where ${\bar 0}, {\bar 1} \in \ZZ_2$. An element $x \in V$ is called {\it homogeneous of degree} $\bar{i}$ if it is an element of $V_{\bar i}, {\bar i} \in \ZZ_2$. In particular, the elements of $V_{\bar 0}$ (resp. $V_{\bar 1}$) are also called {\em even} (resp. {\em odd}). For a homogeneous element $x\in V$ we denote $|x|$ the the degree of $x$ (either ${\bar 0}$ or ${\bar 1}$).

A {\em Lie superalgebra} (see \cite{Kac2}) is a $\ZZ_2$-graded vector space  $\gg=\gg_{\bar 0} \oplus \gg_{\bar 1}$, with an even bilinear commutation operation (or ``supercommutation'') $[\cdot,\cdot]$, which for an arbitrary homogeneous elements $x, y, z$ satisfies the conditions
\begin{enumerate}
\item[1.] $[x,y]=-(-1)^{|x| |y|}[y,x],$

\item[2.]  $(-1)^{|z| |x|}[x,[y,z]] + (-1)^{|x||y|}[y,[z,x]] + (-1)^{|y||z|}[z,[x,y]]=0$ {\it (super Jacobi identity).}
\end{enumerate}

Thus, $\gg_{\bar 0}$ is an ordinary Lie algebra, and $\gg_{\bar 1}$ is a module over $\gg_{\bar 0}$; the Lie superalgebra structure also contains the symmetric pairing $S^2 \gg_{\bar 1} \longrightarrow \gg_{\bar 0}$.

In general, the  {\em descending central sequence}  of a Lie superalgebra $\gg=\gg_{\bar 0} \oplus \gg_{\bar 1}$ is defined in the same way as for Lie algebras:  ${\cal C}^0(\gg): =\gg$, ${\cal C}^{k+1}(\gg):=[{\cal C}^k(\gg),\gg]$  for all $k\geq 0$. Consequently, if ${\cal C}^k(\gg)=\{0\}$ for some $k$, then the Lie superalgebra is
called  {\em nilpotent}.  Then, the smallest integer $k$ such that ${\cal C}^k(\gg)=\{0\}$ is called the {\em nilindex} of the Lie superalgebra $\gg$. Likewise, the {\em derived sequence} of $\gg$ is defined by
${\cal D}^0(\gg): =\gg$, ${\cal D}^{k+1}(\gg):=[{\cal D}^k(\gg),{\cal D}^k(\gg)]$  for all $k\geq 0$. If this sequence is stabilised in zero, then the Lie superalgebra is said to be {\em solvable}. All nilpotent Lie superalgebras are solvable ones.
Engel's theorem and its direct consequences remaind valid for Lie superalgebras. In particular, a Lie superalgebra $L$ is nilpotent if and only if $ad_{L} x$ is nilpotent for every homogeneous element $x$ of $L$. Moreover, for solvable Lie superalgebras we have that a Lie superalgebra $L$ is solvable if and only if its Lie algebra $L_{\overline{0}}$ is solvable. Nevertheless, we do not have the analog of Lie's Theorem and neither its corollaries for solvable Lie superalgebras.

At the same time, there are also defined two other crucial sequences denoted by ${\cal C}^{k}(\gg_{\bar 0})$ and ${\cal C}^{k}(\gg_{\bar 1})$ which will play an important role in our study. They are defined as follows:
$${\cal C}^0(\gg_{\bar i}):=\gg_{\bar i}, {\cal C}^{k+1}(\gg_{\bar i}) := [\gg_{\bar 0}, {\cal C}^k(\gg_{\bar i})], \ k\geq 0, {\bar i} \in \ZZ_2.$$

In the study nilpotent Lie algebras we have the invariant called characteristic sequence that can be naturally extended  for Lie superalgebras.

\begin{defn} For an arbitrary element $x\in \gg_0$, the adjoint operator $ad_x$ is a nilpotent endomorphism of the space $\gg_i$, where $i\in \{0,1\}$. We denote by $gz_i(x)$ the descending sequence of dimensions of
Jordan blocks of $ad_x$. Then, we define the invariant of a Lie
superalgebra $\gg$ as follows:
$$
gz(\gg)=\left(\left.\max_{x\in \gg_0 \setminus[\gg_0,\gg_0]} gz_0(x) \ \right|
\max_{\widetilde{x}\in \gg_0\setminus[\gg_0,\gg_0]} gz_1(\widetilde{x})\right),
$$
where $gz_i$ is in lexicographic order.

The couple $gz(\gg)$ is called characteristic sequence of the Lie superalgebra $\gg$.
\end{defn}

Let us recall the definition of superderivations for Lie superalgebras \cite{Kac2}. A superderivation of degree $s$ of a Lie superalgebra $L$, $s\in \ZZ_2$, is an endomorphism $D \in End_s L$ with the property
$$D(a b)=D(a)b + (-1)^{s \cdot deg a} a D(b).$$

If we denote $Der_{s}L \subset End_s L$ the space of all superderivations of degree $s$, then $Der L=Der _{\overline{0}}L\oplus Der _{\overline{1}}L$ is the Lie superalgebra of superderivations of $L$, with $Der _{\overline{0}}L$ composed by even superderivations and $Der _{\overline{1}}L$ by odd ones.

\subsection{Preliminaries for Leibniz superalgebras.}

\

Remark that many results and definitions of the above section can be extended for Leibniz superalgebras.

\begin{defn}\label{defB} \cite{AO0}. A $\ZZ_2$-graded vector space $L=L_{\bar 0} \oplus L_{\bar 1}$ is called a  {\em Leibniz superalgebra} if it is equipped with a product $[\cdot,\cdot]$ which for an arbitrary element $x$ and homogeneous elements $y, z$ satisfies the condition $$[x,[y,z]]=[[x,y],z]-(-1)^{|y| |z|}[[x,z],y] \hspace{0.4cm} \hbox{{\it (super Leibniz identity)}}.$$
\end{defn}

Note that if a Leibniz superalgebra $L$ satisfies the identity $[x,y]=-(-1)^{|x| |y|}[y,x]$ for any homogeneous elements $x, y \in L$, then the super Leibniz identity  becomes the super Jacobi identity.  Consequently,  Leibniz superalgebras are a generalization of Lie superalgebras. Also and in the same way as for Lie superalgebras, isomorphisms are assumed to be consistent with the $\ZZ_2$-graduation.

Let us now denote  by $R_x$ the right multiplication operator, i.e., $R_x : L \rightarrow L$  given as $R_x(y) := [y,x]$ for $y \in L$, then the super Leibniz identity can be expressed as
$R_{[x,y]} = R_yR_x-(-1)^{|x| |y|}R_xR_y.$

If we denote by $R(L)$ the set of all right multiplication operators, then $R(L)$ with respect to the following multiplication
\begin{equation}\label{product_R(L)}
<R_a,R_b>:=R_aR_b-(-1)^{\bar{i}\bar{j}}R_bR_a
\end{equation}
for $R_a\in {R(L)_{\bar{i}}}$, $R_b\in {R(L)_{\bar{j}}}$, forms a Lie superalgebra.  Note that $R_a$ is a derivation. In fact, the condition for being a derivation of a Leibniz superalgebra (for more details see \cite{null-filiform}) is $d([x,y])=(-1)^{\mid d \mid \mid y \mid} [d(x),y]+[x,d(y)]$.  Since the degree of $R_z$ as homomorphism between $\ZZ_2$-graded vector spaces is the same as the degree of the homogeneous element $z$, that is $\vert R_z\vert=\vert z \vert$, then the condition for $R_z$ to be a derivation is exactly $R_z([x,y])=(-1)^{\mid z \mid \mid y \mid} [R_z(x),y]+[x,R_z(y)]$. This last condition can be rewritten $[[x,y],z]=(-1)^{\mid z \mid \mid y \mid} [[x,z],y]+[x,[y,z]]$ which is nothing but the super (graded) Leibniz identity.

Let us note that the  concepts of descending central sequence, nilindex, the variety of Leibniz superalgebras and Engel's theorem are natural extensions from Lie theory.

Let $V = V_{\bar 0} \oplus V_{\bar 1}$ be the underlying vector space of $L$, $L = L_{\bar 0} \oplus L_{\bar 1} \in Leib^{n,m}$, being $Leib^{n,m}$ the variety of Leibniz superalgebras,  and let $G(V)$ be the group of the invertible linear mappings of the form  $f=f_{\bar 0}+f_{\bar 1}$, such that $f_{\bar 0} \in GL(n,\mathbb{C})$ and $f_{\bar 1} \in GL(m,\mathbb{C})$ (then $G(V) = GL(n,\mathbb{C})\oplus GL(m,\mathbb{C}))$. The action of $G(V)$ on $Leib^{n,m}$ induces an action
on the Leibniz superalgebras variety: two laws  $\lambda_1, \lambda_2$ are {\em isomorphic}  if there exists a linear mapping $f = f_{\bar 0}+f_{\bar 1} \in G(V)$, such that  $$\lambda_2(x,y) = f_{\bar{i}+\bar{j}}^{-1}(\lambda_1(f_{\bar{i}}(x),
f_{\bar{j}}(y))), \hspace{0.2cm} \hbox{for any} \  x \in V_{\bar{i}}, y \in V_{\bar{j}}.$$

Furthermore, the description of the variety of any class of algebras or superalgebras is a difficult problem. Different papers (for example, \cite{Omirov, BS, libroKluwer, GO}) are regarding the applications of algebraic groups theory to the description of the variety of Lie and Leibniz algebras.

\begin{defn}
 For a Leibniz superalgebra $L=L_{\bar 0} \oplus L_{\bar 1}$ we define the {\em right annihilator of} $L$ as the set $Ann(L):=\{x \in L : [L,x]=0\}$.
\end{defn}

It is easy to see that $Ann(L)$ is a two-sided ideal of $L$ and $[x,x] \in Ann(L)$ for any $x \in L_{\bar 0}$. This notion is compatible with the right annihilator in Leibniz algebras. If we consider  the ideal $I:=ideal<[x,y]+(-1)^{ |x| |y|}[y,x]>$, then $I \subset Ann(L).$

Let $L=L_{\bar 0} \oplus L_{\bar 1}$ be a nilpotent Leibniz superalgebra with $\dim L_{\bar 0}=n$ and $\dim L_{\bar 1}=m.$ From Equation \eqref{product_R(L)} we have that $R(L)$ is a Lie superalgebra, and in particular $R(L_{\bar 0})$ is a Lie algebra. As $L_{\bar 1}$ has $L_{\bar 0}$-module structure we can consider $R(L_{\bar 0})$ as a subset of $GL(V_{\bar 1})$  , where  $V_{\bar 1}$  is the underlying vector space of $L_{\bar 1}$. So, we have a Lie algebra formed by nilpotent endomorphisms of $V_{\bar 1}$. Applying the Engel's theorem  we have the existence of a sequence of subspaces of $V_{\bar 1}$,
$V_0 \subset V_1 \subset V_2 \subset \dots \subset V_m = V_{\bar 1},$ with $R(L_{\bar 0})(V_{\overline{i+1}})\subset V_{\bar i}.$ Then, it can be defined the descending sequences $C^k(L_{\bar 0})$ and $C^k(L_{\bar 1})$ and the  super-nilindex  in the same way as for Lie superalgebras. That is,  ${\cal C}^0(L_{\bar i}):=L_{\bar i}, {\cal C}^{k+1}(L_{\bar i}) := [ {\cal C}^k(L_{\bar i}),L_{\bar 0}], \quad k\geq 0, {\bar i} \in \ZZ_2.$ If $L=L_{\bar 0} \oplus L_{\bar 1}$ is a nilpotent Leibniz superalgebra, then $L$ has  {\em super-nilindex} or {\em s-nilindex} $(p,q)$ if satisfies
$${\cal C}^{p-1}(L_{\bar 0}) \neq 0, \qquad {\cal C}^{q-1}(L_{\bar 1})\neq 0, \qquad {\cal C}^{p}(L_{\bar 0})={\cal C}^{q}(L_{\bar 1})=0.$$
We have for Lie superalgebras the invariant called characteristic
sequences that can be naturally extended  for Leibniz
superalgebras. Thus, we have the following definition.

\begin{defn} For an arbitrary element $x\in L_0$, the operator $R_x$ is
a nilpotent endomorphism of the space $L_i$, where $i\in \{0,1\}$.
We denote by $gz_i(z)$ the descending sequences of dimensions of
Jordan blocks of $R_x$. Then, we define the invariant of a Leibniz
superalgebra $L$ as follows:
$$
gz(L)=\left(\left.\max_{x\in L_0 \setminus[L_0,L_0]} gz_0(x) \ \right|
\max_{\widetilde{x}\in L_0\setminus[L_0,L_0]} gz_1(\widetilde{x})\right),
$$
where $gz_i$ is in lexicographic order.

The couple $gz(L)$ is called characteristic sequences of a Leibniz superalgebra $L$.
\end{defn}

\section{Maximal-dimensional solvable Lie superalgebras with filiform nilradical}
Troughout this section we study solvable Lie superalgebras with maximal dimension of the complementary space to nilradical, being the nilradical the model filiform Lie superalgebra $L^{n,m}$. Let us recall that  in \cite{solvableSA} the authors proved that under  the condition of being $L^2$ nilpotent, any solvable Lie superalgebra over the real or complex numbers field can be obtain by means of outer non-nilpotent superderivations of the nilradical. Therefore, for any solvable Lie superalgebra ${\mathfrak{r}}$ with ${\mathfrak{r}}^2$ nilpotent,  we have a decomposition into semidirect sum:
$${\mathfrak{r}}={\mathfrak{t}}\overrightarrow{\oplus} {\mathfrak{n}},$$
$$ [{\mathfrak{t}},{\mathfrak{n}}]\subset {\mathfrak{n}}, \ \ [{\mathfrak{n}},{\mathfrak{n}}]\subset {\mathfrak{n}}, \ \ [{\mathfrak{t}},{\mathfrak{t}}]\subset {\mathfrak{n}}.$$

Along the present section we consider as nilradical the model filiform Lie superalgebra $L^{n,m}$, that is,  the simplest filiform Lie superalgebra which is defined by the only  non-zero products
$$L^{n,m}:
       \left\{\begin{array}{ll}
          [x_1,x_i]=-[x_i,x_1]=x_{i+1},& 2\leq i \leq n-1\\[1mm]
          [x_1,y_j]=-[y_j,x_1]=y_{j+1},& 1\leq j \leq m-1
       \end{array}\right.$$
where $\{ x_1, \dots, x_n \}$ be a basis of $(L^{n,m})_{\bar{0}}$  and $\{ y_1, \dots, y_m \}$ be a basis of $(L^{n,m})_{\bar{1}}$. Note that $L^{n,m}$ is the most important filiform Lie superalgebra, in complete analogy to Lie algebras, since all the other filiform
Lie superalgebras can be obtained from it by deformations \cite{Bor07}. These infinitesimal deformations are given by the even 2-cocycles $Z^2_0(L^{n,m},L^{n,m})$.

On the other hand, ${\mathfrak{t}}=span\{ t_1,t_2,t_3\}$ corresponds with the maximal torus of derivations of $L^{n,m}$, which is composed, in particular, by even superderivations. Then  ${\mathfrak{t}}$ is Abelian ($[{\mathfrak{t}},{\mathfrak{t}}]=0$) and the operators $ad t_i \ (t_i \in {\mathfrak{t}})$ are diagonal. A straightforward computation leads to the following action of ${\mathfrak{t}}$ over $L^{n,m}$:
$$\begin{array}{ll}
[t_1,x_i]=ix_{i}, & 1\leq i \leq n; \\[1mm]
 [t_1,y_j]=jy_{j}, &  1\leq j \leq m; \\[1mm]
[t_2,x_i]=x_{i}, & 2\leq i \leq n; \\[1mm]
[t_3,y_j]=y_{j}, & 1\leq j \leq m.
\end{array}$$

Thus, the solvable Lie superalgebra that we are going to consider, and henceforth named  $SL^{n,m}$, is  defined in a basis  $\{ x_1, \dots, x_n, t_1,t_2,t_3,y_1, \dots, y_m \}$ by the only  non-zero bracket products
$$SL^{n,m}:
       \left\{\begin{array}{ll}
          [x_1,x_i]=-[x_i,x_1]=x_{i+1},& 2\leq i \leq n-1;\\[1mm]
          [x_1,y_j]=-[y_j,x_1]=y_{j+1},& 1\leq j \leq m-1;\\[1mm]
          [t_1,x_i]=-[x_i,t_1]=ix_{i}, & 1\leq i \leq n; \\[1mm]
           [t_1,y_j]=-[y_j,t_1]=jy_{j}, &  1\leq j \leq m; \\[1mm]
[t_2,x_i]=-[x_i,t_2]=x_{i}, & 2\leq i \leq n; \\[1mm]
[t_3,y_j]=-[y_j,t_3]=y_{j}, & 1\leq j \leq m; \\[1mm]
       \end{array}\right.$$
with $\{ x_1, \dots, x_n, t_1,t_2,t_3\}$ a basis of $(SL^{n,m})_{\bar{0}}$  and $\{ y_1, \dots, y_m\}$ a basis of $(SL^{n,m})_{\bar{1}}$. The purpose now is to find out whether $SL^{n,m}$ is the unique solvable Lie superalgebra with  maximal codimension of nilradical $L^{n,m}$.

\begin{thm} \label{solfil} An arbitrary complex maximal-dimensional solvable Lie superalgebra $L$ with $L^2$ nilpotent   and  nilradical  $L^{n,m}$ is isomorphic to $SL^{n,m}$.
\end{thm}
\begin{proof} It is easy to see that superalgebra $L^{n,m}$ can be considered as nilpotent Lie algebra (a $\mathbb{Z}_2$-graded Lie algebra), because of the fact that there is no symmetric bracket products in the law of $L^{n,m}$. Note also, that under the condition of being $L^2$ nilpotent the techniques used in Lie superalgebras are rather similar to the ones used in Lie algebras, i.e. solvable extension by means of (super)derivations of nilradical, for more details it can be consulted \cite{solvableSA}.  Considering now $L^{n,m}$ as a Lie algebra, the results of \cite{UzMathJournal} allow us to assert that there is a unique solvable Lie algebra with maximal codimension of nilradical, i.e. maximal dimension of the complementary space to nilradical. Moreover this maximal codimension is equal to the number of generators of the nilradical,  $3$ in our case.  It can be easily seen that this unique solvable Lie algebra described in \cite{UzMathJournal} is isomorphic to $SL^{n,m}$, considered the latter as a Lie algebra. Indeed, following to Theorem $3.2$  \cite{UzMathJournal} we have a solvable Lie algebra
	
		$$
       \left\{\begin{array}{ll}
          [x_1,x_i]=
          -[x_i,x_1]=
          x_{i+1},& 2\leq i \leq n-1;\\[1mm]
          [x_1,y_j]=
          -[y_j,x_1]=
          y_{j+1},& 1\leq j \leq m-1;\\[1mm]
          [z_1,x_1]=
          -[x_1,z_1]=
          x_1, &\\[1mm]
          [z_1,x_i]=
          -[x_i,z]=
          (i-2)x_{i}, & 3\leq i \leq n; \\[1mm]
          [z_1,y_j]=
          -[y_j,z_1]=
          (j-1)y_{j}, &  2\leq j \leq m; \\[1mm]
[z_2,x_i]=
-[x_i,z_2]=
x_{i}, & 2\leq i \leq n; \\[1mm]
[z_3,y_j]=
-[y_j,z_3]=
y_{j}, & 1\leq j \leq m; \\[1mm]
       \end{array}\right.$$
and the isomorphism defined by $\{t_1=z_1+2z_2+z_3, \ t_2=z_2, \ t_3=z_3\}$ shows that this Lie algebra is isomorphic to $SL^{n,m}$.

Now,  we consider $L^{n,m}$ as a Lie superalgebra. Let us recall now, the definition of superderivations of superalgebras \cite{Kac2}. A superderivation of degree $s$ of a superalgebra $L$, $s\in \ZZ_2$, is an endomorphism $D \in End_s L$ with the property
$$D([a, b])=[D(a),b] + (-1)^{s \cdot deg a} [a, D(b)]$$

Thus, the even superderivations $D$ of $L^{n,m}$ are in particular derivations of the $\mathbb{Z}_2$-graded Lie algebra $L^{n,m}$ verifying $D(L^{n,m}_{\overline{0}})\subset L^{n,m}_{\overline{0}}$ and $D(L^{n,m}_{\overline{1}})\subset L^{n,m}_{\overline{1}}$. The odd superderivations $D$ of $L^{n,m}$, on the other hand, verify  the following three conditions:
\begin{eqnarray}
 \label{odder1}D([x_i,x_j])=[D(x_i),x_j]+[x_i,D(x_j)],  & \mbox{ if } x_i,x_j \in L^{n,m}_{\overline{0}}\\[1mm]
 \label{odder2}D([x_i,y_j])=[D(x_i),y_j]+[x_i,D(y_j)],  & \mbox{ if } x_i \in L^{n,m}_{\overline{0}},  y_j \in L^{n,m}_{\overline{1}}\\[1mm]
\label{odder3}D([y_i,y_j])=[D(y_i),y_j]-[y_i,D(y_j)],  & \mbox{ if } y_i,y_j \in L^{n,m}_{\overline{1}}
 \end{eqnarray}

 The equations (\ref{odder1}) and (\ref{odder2}) correspond with the Lie derivation condition. Thus, it remains to study the equation (\ref{odder3}). Taking into account now the law of $L^{n,m}$ it can be easily seen that the only possibility for having at least one non-null term in the equation (\ref{odder3}) corresponds with the existence of $y_i$ such that $x_1 \in D(y_i)$. Next, we explore such possibilities. For $i\geq 2$, we get
 $$D(y_i)=D([x_1,y_{i-1}])=[D(x_1),y_i]+[x_1,D(y_{i-1})]=[x_1,D(y_{i-1})] $$

Since $[x_1,D(y_{i-1})] \in span\{ x_3, \dots,x_n\}$ we can exclude $i$ for $i\geq 2$. Suppose $i=1$, and then $$D(y_1)=\alpha_1 x_1+\displaystyle \sum_{k=2}^{n} \alpha_k x_k, \qquad \mbox{ with } \alpha_1 \neq 0.$$

From equation (\ref{odder3}) we get in particular
$$D([y_1,y_1])=0=[D(y_1),y_1]-[y_1,D(y_1)]=2\alpha_1 y_2$$
and then $\alpha_1=0$. Therefore the equation  (\ref{odder3}) vanishes over $L^{n,m}$ and consequently all the odd superderivations of $L^{n,m}$ are  in particular derivations of the $\mathbb{Z}_2$-graded Lie algebra $L^{n,m}$ verifying $D(L^{n,m}_{\overline{0}})\subset L^{n,m}_{\overline{1}}$ and $D(L^{n,m}_{\overline{1}})\subset L^{n,m}_{\overline{0}}$.

Thus,  all the superderivations of the Lie superalgebra $L^{n,m}$ are particular cases of derivations of the $\mathbb{Z}_2$-graded Lie algebra $L^{n,m}$. Therefore, we can assert that $SL^{n,m}$ is not only the  unique maximal-dimensional solvable Lie algebra with  nilradical the Lie algebra $L^{n,m}$, but also is the  unique maximal-dimensional solvable Lie superalgebra $L$ with $L^2$ nilpotent and  nilradical the model filiform Lie superalgebra $L^{n,m}$.
\end{proof}

\section{Maximal-dimensional solvable Lie superalgebras with model nilpotent  nilradical}
Throughout this section firstly we extend  the definition of model nilpotent to Lie superalgebras and after that we obtain the description of maximal-dimensional solvable Lie superalgebra with model nilpotent  nilradical.

Next we recall the definition of model nilpotent Lie algebra, for more details it can be consulted for instance \cite{rigid_algebras}. Thus,  the model nilpotent Lie algebra with arbitrary characteristic sequence $(n_1,n_2, \cdots,n_k,1)$  is composed by the Lie algebras admitting a basis $\{x_1, \cdots, x_{n_1+1}, \cdots, x_{n_1+n_2+1}, $ $\cdots, x_{n_1+\cdots n_k+1} \}$ such that the only non-null brackets (expanded by skew-symmetry) are exactly the following

$$\begin{array}{rcl}
[x_1,x_j]&=&x_{j+1}, \ 2\leq j \leq n_1, \\[1mm]
[x_1,x_{n_1+j}]&=&x_{n_1+1+j}, \ 2\leq j \leq n_2, \\[1mm]
\vdots & & \\ [1mm]
[x_1,x_{n_1+\cdots +n_{k-2}+j}]&=&x_{n_1+\cdots+n_{k-2}+1+j}, \ 2\leq j \leq n_{k-1}, \\[1mm]
[x_1,x_{n_1+\cdots +n_{k-1}+j}]&=&x_{n_1+\cdots+n_{k-1}+1+j}, \ 2\leq j \leq n_{k}. \\[1mm]
\end{array}$$

Therefore, we present the next definition in a natural way.

\begin{defn} The model nilpotent Lie superalgebra  with arbitrary characteristic sequences $(n_1, \cdots,n_k,1 | m_1, \cdots, m_p)$  is nothing but the Lie superalgebra admitting a basis $\{x_1, $$\cdots,$ $
x_{n_1+\cdots n_k+1},$ $ y_1, \cdots,y_{m_1+\cdots+m_p} \}$ with $x_i$ even basis vectors and $y_j$ odd basis vectors, such that the only non-null brackets are exactly the following

\

$N(n_1, \cdots,n_k,1 | m_1, \cdots, m_p):$
$$\left\{\begin{array}{ll}
[x_1,x_j]=
-[x_j,x_1]=
x_{j+1}, & 2\leq j \leq n_1, \\[1mm]
[x_1,x_{n_1+\cdots+n_j+i}]=
-[x_{n_1+\cdots+n_j+i},x_1]=
x_{n_1+\cdots+n_j+i+1}, & 1\leq j \leq k-1, \ 2 \leq i \leq n_{j+1},\\[1mm]
[x_1,y_j]=
-[y_j,x_1]=
y_{j+1}, & 1\leq j \leq m_1-1, \\[1mm]
[x_1,y_{m_1+\cdots+m_j+i}]=
-[y_{m_1+\cdots+m_j+i},x_1]=
y_{m_1+\cdots+m_j+i+1}, & 1\leq j \leq p-1, \ 1 \leq i \leq m_{j+1}-1.\\[1mm]
\end{array}\right.$$
\end{defn}

\begin{rem} Note that the model filiform Lie superalgebra $L^{n,m}$ is exactly  $N(n-1,1 | m)$.

\end{rem}

Let us consider now ${\mathfrak{t}}=span\{ t_1,\dots,t_{k+1},t'_1,\dots,t'_p \}$  the maximal torus of derivations of $N_{(n_1, \cdots,n_k,1 | m_1, \cdots, m_p)}$, which is composed in particular by even superderivations.  A straightforward computation leads to the following action of ${\mathfrak{t}}$ over $N(n_1, \cdots,n_k,1 | m_1, \cdots, m_p):$
$$\begin{array}{ll}
[t_1,x_i]=ix_{i}, & 1\leq i \leq  n_1+\cdots+n_k+1; \\[1mm]
 [t_1,y_j]=jy_{j}, &  1\leq j \leq m_1+\cdots+m_p; \\[1mm]
[t_2,x_i]=x_{i}, & 2 \leq i \leq n_1+1; \\[1mm]
[t_{j+2},x_{n_1+\cdots+n_j+i}]=x_{n_1+\cdots+n_j+i}, &  1 \leq j \leq k-1, \  2 \leq i \leq n_{j+1}+1  ; \\[1mm]
[t'_1,y_i]=y_{i}, & 1\leq i \leq m_1; \\[1mm]
[t'_{j+1},y_{m_1+\cdots+m_j+i}]=y_{m_1+\cdots+m_j+i}, &  1 \leq j \leq p-1, \  1 \leq i \leq m_{j+1}. \\[1mm]
\end{array}$$

Thus, the solvable Lie superalgebra that we are going to consider and denoted by $SN(n_1, \cdots,n_k,1 |$ $ m_1, \cdots, m_p)$ is  defined in a basis  $\{ x_1, \dots, x_{n_1+\cdots n_k+1}, $ $t_1,\dots,t_{k+1},$ $t'_1,\dots,t'_p, $ $y_1, \dots, y_{m_1+\cdots+m_p} \}$ by the only  non-zero bracket products:
$SN(n_1, \cdots,n_k,1 | m_1, \cdots, m_p):$
$$\left\{\begin{array}{ll}
[x_1,x_j]=
-[x_j,x_1]=
x_{j+1}, & 2\leq j \leq n_1; \\[1mm]
[x_1,x_{n_1+\cdots+n_j+i}]=
-[x_{n_1+\cdots+n_j+i},x_1]=
x_{n_1+\cdots+n_j+i+1}, & 1\leq j \leq k-1, \ 2 \leq i \leq n_{j+1};\\[1mm]
[x_1,y_j]=
-[y_j,x_1]=
y_{j+1}, & 1\leq j \leq m_1-1; \\[1mm]
[x_1,y_{m_1+\cdots+m_j+i}]=
-[y_{m_1+\cdots+m_j+i},x_1]=
y_{m_1+\cdots+m_j+i+1}, & 1\leq j \leq p-1, \ 1 \leq i \leq m_{j+1}-1;\\[1mm]
[t_1,x_i]=
-[x_i,t_1]=
ix_{i}, & 1\leq i \leq  n_1+\cdots+n_k+1; \\[1mm]
[t_1,y_j]=
-[y_j,t_1]=
jy_{j}, &  1\leq j \leq m_1+\cdots+m_p; \\[1mm]
[t_2,x_i]=
-[x_i,t_2]=
x_{i}, & 2 \leq i \leq n_1+1; \\[1mm]
[t_{j+2},x_{n_1+\cdots+n_j+i}]=
-[x_{n_1+\cdots+n_j+i},t_{j+2}]=
x_{n_1+\cdots+n_j+i}, &  1 \leq j \leq k-1, \  2 \leq i \leq n_{j+1}+1  ; \\[1mm]
[t'_1,y_i]=
-[y_i,t'_1]=
y_{i}, & 1\leq i \leq m_1; \\[1mm]
[t'_{j+1},y_{m_1+\cdots+m_j+i}]=
-[y_{m_1+\cdots+m_j+i},t'_{j+1}]=
y_{m_1+\cdots+m_j+i}, &  1 \leq j \leq p-1, \  1 \leq i \leq m_{j+1};
\end{array}\right.$$

\noindent with $\{ x_1, \dots, x_{n_1+\cdots n_k +1},$ $ t_1,\dots,t_{k+1},$ $t'_1,\dots,t'_p\}$ even basis vectors and $\{ y_1, \dots, $ $y_{m_1+\cdots+m_p}\}$ odd basis vectors. Next, we show that this solvable Lie superalgebra is the unique with  maximal codimension of nilradical $N(n_1, \cdots,n_k,1 | m_1, \cdots, m_p)$.

\begin{thm} Let $L$ be a complex maximal-dimensional solvable Lie superalgebra with $L^2$ nilpotent and nilradical  isomorphic to $N(n_1, \cdots,n_k,1 | m_1, \cdots, m_p)$. Then there exists a basis, namely $\{ x_1, \dots, x_{n_1+\cdots n_k+1}, $ $t_1,\dots,t_{k+1},$ $t'_1,\dots,t'_p, $ $y_1, \dots, y_{m_1+\cdots+m_p} \}$ with $\{ x_1, \dots, x_{n_1+\cdots n_k +1},$ $ t_1,\dots,t_{k+1},$ $t'_1,\dots,t'_p\}$ even basis vectors and $\{ y_1, \dots, $ $y_{m_1+\cdots+m_p}\}$ odd basis vectors, in which $L$ is isomorphic to $SN(n_1, \cdots,n_k,1 | m_1, \cdots, m_p)$.
\end{thm}

\begin{proof} On account of the lack of symmetric bracket products, $N(n_1, \cdots,n_k,1 | m_1, \cdots, m_p)$  can be regarded as both  a nilpotent Lie superalgebra and a nilpotent $\mathbb{Z}_2$-graded Lie algebra. Likewise, under the condition of being $L^2$ nilpotent the techniques used in Lie superalgebras are  similar to the ones used in Lie algebras, that is solvable extension by means of (super)derivations of nilradical.

Let us consider now $N(n_1, \cdots,n_k,1 | m_1, \cdots, m_p)$ as a Lie algebra, the results of \cite{UzMathJournal} allow us to assert that there is only one solvable Lie algebra with maximal codimension of nilradical, which can be expressed in a suitable basis $\{ x_1, \dots, x_{n_1+\cdots n_k+1}, $ $z_1,\dots,z_{k+1},$ $z'_1,\dots,z'_p, $ $y_1, \dots, y_{m_1+\cdots+m_p} \}$, by the only non zero bracket products:
$$\left\{\begin{array}{ll}
[x_1,x_j]=
-[x_j,x_1]=
x_{j+1}, & 2\leq j \leq n_1; \\[1mm]
[x_1,x_{n_1+\cdots+n_j+i}]=
-[x_{n_1+\cdots+n_j+i},x_1]=
x_{n_1+\cdots+n_j+i+1}, & 1\leq j \leq k-1, \ 2 \leq i \leq n_{j+1};\\[1mm]
[x_1,y_j]=
-[y_j,x_1]=
y_{j+1}, & 1\leq j \leq m_1-1; \\[1mm]
[x_1,y_{m_1+\cdots+m_j+i}]=
-[y_{m_1+\cdots+m_j+i},x_1]=
y_{m_1+\cdots+m_j+i+1}, & 1\leq j \leq p-1, \ 1 \leq i \leq m_{j+1}-1;\\[1mm]
[z_1,x_1]=
-[x_1,z_1]=
x_1; &\\[1mm]
[z_1,x_i]=
-[x_i,z_1]=
(i-2)x_{i}, & 3\leq i \leq n_1+1; \\[1mm]
[z_1,x_{n_1+\cdots+n_j+i}]=
-[x_{n_1+\cdots+n_j+i},z_1]=
(i-2)x_{n_1+\cdots+n_j+i}, & 1\leq j \leq k-1, \ 3 \leq i \leq n_{j+1}+1;\\[1mm]
[z_1,y_i]=
-[y_i,z_1]=
(i-1)y_{i}, &  2\leq i \leq m_1; \\[1mm]
[z_1,y_{m_1+\cdots+m_j+i}]=
-[y_{m_1+\cdots+m_j+i},z_1]=
(i-1)y_{m_1+\cdots+m_j+i}, & 1\leq j \leq p-1, \ 2 \leq i \leq m_{j+1};\\[1mm]
[z_2,x_i]=
-[x_i,z_2]=
x_{i}, & 2 \leq i \leq n_1+1; \\[1mm]
[z_{j+2},x_{n_1+\cdots+n_j+i}]=
-[x_{n_1+\cdots+n_j+i},z_{j+2}]=
x_{n_1+\cdots+n_j+i}, &  1 \leq j \leq k-1, \  2 \leq i \leq n_{j+1}+1;\\[1mm]
[z'_1,y_i]=
-[y_i,z'_1]=
y_{i}, & 1\leq i \leq m_1; \\[1mm]
[z'_{j+1},y_{m_1+\cdots+m_j+i}]=
-[y_{m_1+\cdots+m_j+i},z'_{j+1}]=
y_{m_1+\cdots+m_j+i}, &  1 \leq j \leq p-1, \  1 \leq i \leq m_{j+1};  \\[1mm]
\end{array}\right.$$

\noindent and the isomorphism defined by

$t_1=z_1+2z_2+\left(\displaystyle \sum_{j=1}^{k-1} (n_1+\cdots+n_j+2)z_{j+2}\right)+z'_1+\left(\displaystyle \sum_{j=1}^{p-1} (m_1+\cdots+m_j+1)z'_{j+1}\right),$

$t_i=z_i, \ 2\leq i \leq k+1 \quad \mbox{ and }\quad  t'_i=z'_i, \ 1\leq i \leq p,$

\

\noindent shows that this Lie algebra is isomorphic to the $\mathbb{Z}_2$-graded Lie algebra $SN(n_1, \cdots,n_k,1 | m_1, \cdots, m_p)$.

Let us consider now $N(n_1, \cdots,n_k,1 | m_1, \cdots, m_p)$ as a Lie  superalgebra.
Analogously as it was seen along the proof of Theorem \ref{solfil}, the even superderivations
are in particular Lie derivations, and for odd superderivations the only condition different from Lie derivation condition is exactly equation (\ref{odder3}):
$$D([y_i,y_j])=[D(y_i),y_j]-[y_i,D(y_j)]$$

 On account of the law of $N(n_1, \cdots,n_k,1 | m_1, \cdots, m_p)$ it can be easily seen that the only possibility for having at least one non-null term in the equation (\ref{odder3}) corresponds with the existence of $y_i$ such that $x_1 \in D(y_i)$. Next, we explore such possibilities. For $y_i$ different from the odd generator vectors, i.e. $i \notin \{1, m_1+1, m_1+m_2+1, \dots, m_1+\cdots+m_{p-1}+1 \}$, we get
 $$D(y_i)=D([x_1,y_{i-1}])=[D(x_1),y_i]+[x_1,D(y_{i-1})]=[x_1,D(y_{i-1})] $$

As $[x_1,D(y_{i-1})] \in span\{ x_3, \dots,x_n\}$ we can exclude $i$ for $i \notin \{1, m_1+1,  \dots, m_1+\cdots+m_{p-1}+1 \}$. Suppose now that there exists $i$, $i \in \{1, m_1+1, , \dots, m_1+\cdots+m_{p-1}+1 \}$ such that  $$D(y_i)=\alpha_1 x_1+\displaystyle \sum_{k=2}^{n_1+\cdots+n_k +1} \alpha_k x_k, \qquad \mbox{ with } \alpha_1 \neq 0$$
From equation (\ref{odder3}) we get in particular
$$D([y_i,y_i])=0=[D(y_i),y_i]-[y_i,D(y_i)]=2\alpha_1 y_{i+1}$$
and then $\alpha_1=0$. Therefore the equation  (\ref{odder3}) vanishes over
$N(n_1, \cdots,n_k,1 | m_1, \cdots, m_p)$ and consequently all the odd superderivations of $N(n_1, \cdots,n_k,1 | m_1, \cdots, m_p)$ are  in particular Lie derivations of itself regarded as  $\mathbb{Z}_2$-graded Lie algebra.

Thus,  all the superderivations of the Lie superalgebra $N(n_1, \cdots,n_k,1 | m_1, \cdots, m_p)$ are particular cases of Lie derivations of itself regarded as  $\mathbb{Z}_2$-graded Lie algebra. Therefore, we can assert that $SN(n_1, \cdots,n_k,1 | m_1, \cdots, m_p)$ is not only the  unique maximal-dimensional solvable Lie algebra with  nilradical the Lie algebra $N(n_1, \cdots,n_k,1 | m_1, \cdots, m_p)$, but also is the  unique maximal-dimensional solvable Lie superalgebra $L$ with $L^2$ nilpotent and  nilradical the model nilpotent Lie superalgebra.
\end{proof}

\section{Maximal-dimensional solvable Leibniz superalgebras with non-Lie filiform nilradical.}

In this section, we consider Leibniz superalgebra whose nilradical is isomorphic to the filiform (non-Lie) Leibniz superalgebra. This filiform Leibniz superalgebra (denoted by $LP^{n,m}$) can be expressed by the only non-null bracket products that follow:
$$\left\{\begin{array}{ll}
         [x_i,x_1]=x_{i+1},& 2\leq i \leq n-1\\[1mm]
          [y_j,x_1]=y_{j+1},& 1\leq j \leq m-1
       \end{array}\right.$$

\begin{thm} \label{solLeib} Let $L$ be a complex maximal-dimensional solvable Leibniz superalgebra with $L^2$ nilpotent  and  with nilradical isomorphic to $LP^{n,m}$. Then there exists a basis, namely $\{ x_1, \dots, x_n, $ $t_1,t_2,t_3,$ $y_1, \dots, y_m \}$ with $\{ x_1, \dots, x_n, $ $t_1,t_2,t_3\}$  a basis of $L_{\overline{0}}$ and $\{ y_1, \dots, y_m\}$  a basis of $L_{\overline{1}}$, in which $L$ is isomorphic to  the following solvable Leibniz superalgebra:
	$$SLP^{n,m}:\left\{\begin{array}{ll}
	[x_i,x_1]=x_{i+1},&2\leq i\leq n-1;\\{}
	[y_j,x_1]=y_{j+1},&1\leq j\leq m-1;\\{}
	[t_1,x_1]=-x_1,&\\{}
	[x_1,t_1]=x_1,&\\{}
	[x_i,t_1]=(i-2)x_i,&3\leq i\leq n;\\{}
	[y_j,t_1]=(j-1)y_j,&2\leq j\leq m;\\{}
	[x_i,t_2]=x_i,&2\leq i\leq n;\\{}
	[y_j,t_3]=y_j,&1\leq j\leq m;
	\end{array}\right.$$
	where the omitted products are zero.
\end{thm}

\begin{proof} Note that $LP^{n,m}$ can be considered as $\mathbb{Z}_2$-graded nilpotent Leibniz algebra. Similar to  Lie case, we can use the result  of paper  \cite{Abdurasulov}, which give the description of solvable Leibniz algebras with maximal codimension of nilradical (maximal dimension of the complementary space to nilradical). This maximal codimension is equal to the number of generators of the nilradical, 3 in our case.
Then, using Theorem $4$ of \cite{Abdurasulov} we have the following products:
	 $$\begin{array}{ll}
	 [x_i,x_1]=x_{i+1},&2\leq i\leq n-1, \\{}
	 [y_j,x_1]=y_{j+1},&1\leq j\leq m-1, \\{}
	 [t_1,x_1]=(b_1-1)x_1,&\\{}
	 [t_2,x_2]=(b_2-1)x_2,&\\{}
	 [t_3,y_1]=(b_3-1)y_1,& \\{}
	 [x_1,t_1]=x_1,&\\{}
	 [x_i,t_1]=(i-2)x_i,&3\leq i\leq n,\\{}
	 [y_j,t_1]=(j-1)y_j,&2\leq j\leq m,\\{}
	 [x_i,t_2]=x_i,&2\leq i\leq n,\\{}
	 [y_j,t_3]=y_j,&1\leq j\leq m,
 \end{array}$$
with $b_i\in \{0,1\}$, $1\leq i\leq 3.$ Only rest to determine the following products $[t_k,x_i],$ with $3\leq i\leq n$, $[t_k,y_j],$ with $2\leq j\leq m$ and $1\leq k\leq 3$. Using the Leibniz identity and the induction method we derive the remaining products of $SLP^{n,m}$:
$$\begin{array}{lll}
[t_1,x_i]=[t_1,y_j]=0,& &2\leq i\leq n,\ 1\leq j\leq m,\\{}
[t_2,x_i]=(b_2-1)x_i,&[t_2,y_j]=0,&3\leq i\leq n,\ 2\leq j\leq m,\\{}
[t_3,x_i]=0,&[t_3,y_j]=(b_3-1)y_j,&3\leq i\leq n,\ 2\leq j\leq m.\\{}

\end{array}$$

Finally, from Leibniz identity for the triples $\{x_2,t_1,x_1\},$ $\{t_2,t_1,x_3\}$ and $\{t_3,x_1,y_1\}$ we obtain $b_1=0,$ $b_2=1$ and $b_3=1,$ respectively.

Now we consider $LP^{n,m}$ as a Leibniz superalgebra. Recall that $d$ is a Leibniz superderivation on $LP^{n,m}$ if $d$ verifies the condition:
$$d([x,y])=(-1)^{|d||y|}[d(x),y]+[x,d(y)]$$

Analogously to Lie superalgebra, the even superderivations of $LP^{n,m}$  are in particular Leibniz derivations verifying that $d(LP_{\overline{0}}^{n,m})\subset LP_{\overline{0}}^{n,m}$ and $d(LP_{\overline{1}}^{n,m})\subset LP_{\overline{1}}^{n,m}$. On the other hand, the odd superderivations $d$ of $LP^{n,m}$ verify the following:
\begin{eqnarray}
\label{odderL1}d([x_i,x_j])=[d(x_i),x_j]+[x_i,d(x_j)],  & x_i,x_j \in LP^{n,m}_{\overline{0}};\\[1mm]
\label{odderL2}d([x_i,y_j])=-[d(x_i),y_j]+[x_i,d(y_j)],  & x_i \in LP^{n,m}_{\overline{0}},  y_j \in LP^{n,m}_{\overline{1}};\\[1mm]
\label{odderL3}d([y_j,x_i])=[d(y_j),x_i]+[y_j,d(x_i)],  & x_i \in LP^{n,m}_{\overline{0}},  y_j \in LP^{n,m}_{\overline{1}};\\[1mm]
\label{odderL4}d([y_i,y_j])=-[d(y_i),y_j]+[y_i,d(y_j)],  & y_i,y_j \in LP^{n,m}_{\overline{1}}.
\end{eqnarray}

Let $d$ be an odd superderivation. Then, we have that
$$d(x_1)=\sum_{k=1}^m a_k y_k,\quad d(x_2)=\sum_{k=1}^m b_k y_k,\quad d(y_1)=\sum_{k=1}^n c_k x_k. $$

Using the equation $(\ref{odderL1})$ for the pair $[x_1,x_1]$ we get $a_i=0$ with $1\leq i\leq m-1$ and from $(\ref{odderL4})$ for the pair $[y_1,y_1]$ we obtain $c_1=0.$ From the equations $(\ref{odderL1})$ and $(\ref{odderL3})$ we compute $d(x_i)$ with $3\leq i \leq n$ and $d(y_j)$ for $2\leq j\leq m.$ Thus, we have:
$$d(x_1)=a_m y_m, \quad d(x_i)=\displaystyle\sum_{k=i-1}^m b_{k-i+2}y_k, \quad 2\leq i\leq n, \quad
d(y_j)=\displaystyle\sum_{k=j+1}^n c_{k-j+1} x_k, \quad 1\leq j\leq m.$$

Finally, from the equation $(\ref{odderL1})$ for the pair $[x_n,x_1]$ we have that $b_k=0$ for $1\leq k\leq {m-n+2}$ if $m\geq n-1.$

It is easy to prove that all odd derivations are Leibniz derivations because the equations $(\ref{odderL2})$ and $(\ref{odderL4})$ vanish.

The equations (\ref{odderL1}) and  (\ref{odderL3})  correspond with Leibniz derivation condition. Then we can reason as in Theorem \ref{solfil} and we get that all odd superderivations of $LP^{n,m}$ are in particular derivations of the $\mathbb{Z}_2$-graded Leibniz algebra $LP^{n,m}.$ Then we can assert that $SLP^{n,m}$ is the  unique maximal-dimensional solvable Leibniz superalgebra $L$ with $L^2$ nilpotent and nilradical the Leibniz superalgebra $LP^{n,m}.$
\end{proof}

\section{Maximal-dimensional solvable Leibniz superalgebras with model nilpotent non-Lie nilradical}

In this section, we consider as nilradical the equivalent of the model nilpotent Lie superalgebra into (non-Lie) Leibniz superalgebras.  This Leibniz superalgebra denoted by $NP(n_1, \cdots,n_k,1 |$ $ m_1, \cdots, m_p)$  can be expressed by the only non-null bracket products that follow:
$$\left\{\begin{array}{ll}
[x_j,x_1]=x_{j+1}, & 2\leq j \leq n_1, \\[1mm]
[x_{n_1+\cdots+n_j+i},x_1]=x_{n_1+\cdots+n_j+i+1}, & 1\leq j \leq k-1, \ 2 \leq i \leq n_{j+1},\\[1mm]
[y_j,x_1]=y_{j+1}, & 1\leq j \leq m_1-1, \\[1mm]
[y_{m_1+\cdots+m_j+i},x_1]=y_{m_1+\cdots+m_j+i+1}, & 1\leq j \leq p-1, \ 1 \leq i \leq m_{j+1}-1.\\[1mm]
\end{array}\right.$$

\begin{thm}
Let $L$ be a complex maximal-dimensional solvable Leibniz superalgebra with $L^2$ nilpotent and nilradical  isomorphic to $NP(n_1, \cdots,n_k,1 | m_1, \cdots, m_p)$. Then there exists a basis, namely $\{ x_1, \dots, x_{n_1+\cdots n_k+1}, $ $t_1,\dots,t_{k+1},$ $t'_1,\dots,t'_p, $ $y_1, \dots, y_{m_1+\cdots+m_p} \}$ with $\{ x_1, \dots,$ $ x_{n_1+\cdots n_k +1},$ $ t_1,\dots,t_{k+1},$ $t'_1,\dots,t'_p\}$ even basis vectors and $\{ y_1, \dots, $ $y_{m_1+\cdots+m_p}\}$ odd basis vectors, in which $L$ is isomorphic to $SNP(n_1, \cdots,n_k,1 | m_1, \cdots, m_p)$ given by:

$$\left\{\begin{array}{ll}
[x_j,x_1]=x_{j+1}, & 2\leq j \leq n_1; \\[1mm]
[x_{n_1+\cdots+n_j+i},x_1]=x_{n_1+\cdots+n_j+i+1}, & 1\leq j \leq k-1, \ 2 \leq i \leq n_{j+1};\\[1mm]
[y_j,x_1]=y_{j+1}, & 1\leq j \leq m_1-1; \\[1mm]
[y_{m_1+\cdots+m_j+i},x_1]=y_{m_1+\cdots+m_j+i+1}, & 1\leq j \leq p-1, \ 1 \leq i \leq m_{j+1}-1;\\[1mm]
[t_1,x_1]=-x_1,&\\[1mm]
[x_1,t_1]=x_1,&\\[1mm]
[x_i,t_1]=(i-2)x_i,& 3\leq i\leq n_1+1;\\[1mm]
[x_{n_1+\dots+n_j+i},t_1]=(i-2)x_{n_1+\dots+n_j+i},&1\leq j\leq k-1,\ 3\leq i\leq n_{j+1}+1;\\[1mm]
[y_j,t_1]=(i-1)y_j,&2\leq j\leq m_1;\\[1mm]
[y_{m_1+\dots+m_j+i},t_1]=(i-1)y_{m_1+\dots+m_j+i}, &1\leq j\leq p-1,\ 2\leq i\leq m_{j+1},\\[1mm]
[x_i,t_2]=x_i,&2\leq i\leq n_1+1;\\[1mm]
[x_{n_1+\dots+n_j+i},t_{j+2}]=x_{n_1+\dots+n_j+i},&1\leq j\leq k-1,\ 2\leq i\leq n_{j+1}+1;\\[1mm]
[y_j,t'_1]=y_j,&1\leq j\leq m_1;\\[1mm]
[y_{m_1+\cdots+m_j+i},t'_{j+1}]=y_{m_1+\cdots+m_j+i},& 1\leq j \leq p-1, \ 1 \leq i \leq m_{j+1};
\end{array}\right.$$
where the omitted products are zero.
\end{thm}

\begin{proof}
	
Similar to Theorem \ref{solLeib}, we consider the superalgebra $NP(n_1, \cdots,n_k,1 | m_1, \cdots, m_p)$ as a nilpotent Leibniz algebra (a $\mathbb{Z}_2$-graded Leibniz algebra). Thus, we can use the results of the paper \cite{Abdurasulov} and we obtain the following products:
$$\begin{array}{ll}
[t_1,x_1]=(b_1-1)x_1,&\\[1mm]
[x_1,t_1]=x_1,&\\[1mm]
[x_i,t_1]=(i-2)x_i,& 3\leq i\leq n_1+1,\\[1mm]
[x_{n_1+\dots+n_j+i},t_1]=(i-2)x_{n_1+\dots+n_j+i},&1\leq j\leq k-1,\ 3\leq i\leq n_{j+1}+1,\\[1mm]
[y_j,t_1]=(i-1)y_j,&2\leq j\leq m_1,\\[1mm]
[y_{m_1+\dots+m_j+i},t_1]=(i-1)y_{m_1+\dots+m_j+i}, &1\leq j\leq p-1,\ 2\leq i\leq m_{j+1},\\[1mm]
[t_2,x_2]=(b_2-1)x_2, &\\[1mm]
[x_2,t_2]=x_2,&\\[1mm]
[t_{j+2},x_{n_1+\dots+n_j+2}]=(b_{j+2}-1)x_{n_1+\dots+n_j+2},&\\[1mm]
[x_{n_1+\dots+n_j+2},t_{j+2}]=x_{n_1+\dots+n_j+2},&\\[1mm]
[y_1,t'_1]=y_1,&\\[1mm]
[t'_1,y_1]=(b'_1-1)y_1,&\\[1mm]
[t'_{j+1},y_{m_1+\cdots+m_j+1}]=(b'_{j+1}-1)y_{m_1+\cdots+m_j+1},&\\[1mm]
[y_{m_1+\cdots+m_j+1},t'_{j+1}]=y_{m_1+\cdots+m_j+1}&
\end{array}$$

Similar to previous case and applying Leibniz identity we obtain the remaining products and 
$b_1=1,$ $b_i=0$ with $2\leq i\leq k+1$ and $b'_j=0$ with $1\leq j\leq p.$ Thus,
we get $SNP(n_1, \cdots,n_k,1 | m_1, \cdots, m_p)$.

Now, we consider $NP(n_1, \cdots,n_k,1 | m_1, \cdots, m_p)$ as a superalgebra. The even superderivations of this superalgebra are in particular Leibniz derivations. Then, we go to prove that the odd superderivations are also Leibniz derivations. For that purpose, it is sufficient to verify that the equations $(\ref{odderL2})$ and $(\ref{odderL4})$ vanish.

Let $d$ be an odd superderivation. Then, we have
$$d(x_1)=\displaystyle\sum_{k=1}^{m_1+\dots+m_p} a_k y_k, \ \ d(x_2)=\displaystyle\sum_{k=1}^{m_1+\dots+m_p} b_k y_k, \ \  d(x_{n_1+\dots+n_j+2})=\displaystyle\sum_{k=1}^{m_1+\dots+m_p} \alpha_{k\, j} y_k, \ 1\leq j\leq k-1,$$
$$d(y_1)=\displaystyle\sum_{t=1}^{n_1+\dots+n_k+1} c_t x_t, \quad d(y_{m_1+\dots+m_j+1})=\displaystyle\sum_{t=1}^{n_1+\dots+n_k+1} \beta_{t\, j} x_t, \quad 1\leq j\leq p.$$

From the equation $(\ref{odderL1})$  for the pair $[x_1,x_1]$ we have that $$d(x_1)=a_{m_1}y_{m_1}+\dots+a_{m_1+\dots+m_p}y_{m_1+\dots+m_p}$$
and for the eqution $(\ref{odderL4})$ for the pair $[y_1,y_1]$ we get to $c_1=0.$ Analogously, if we consider  the pairs $[y_{m_1+\dots+m_j+1},y_{m_1+\dots+m_j+1}]$ with $1\leq j\leq p-1$ we have that $\beta_{1\, j}=0$. Thus, it proves that the equation $(\ref{odderL4})$ is always zero.

On the other hand, we put the equation $(\ref{odderL2})$. The produts $[x_i,d(y_j)]$ are always zero because of $c_1=\beta_{1\, j}=0$ for $1\leq j\leq p-1$ in $d(y_j)$. The other products are trivially zero.

Finally, we conclude that the odd superderivations are in particular Leibniz derivations. We can assert that $SNP(n_1, \cdots,n_k,1 | m_1, \cdots, m_p)$ is the  unique maximal-dimensional solvable Leibniz superalgebra $L$ with $L^2$ nilpotent and  nilradical the model nilpotent Leibniz superalgebra $NP(n_1, \cdots,n_k,1 | m_1, \cdots, m_p)$.

\end{proof}

\section{Superderivations of  the maximal-dimensional solvable Lie and Leibniz superalgebras}

In this section we establish that the spaces of superderivations for the superalgebras obtained in prevous sections (that is, $SL^{n,m}$, $SN(n_1, \cdots,n_k,1 | m_1, \cdots, m_p), SLP^{n,m}, SNP(n_1, \cdots,n_k,1 | m_1, $ $\cdots, m_p)$ consist of inner superderivations. These results extend analogously results for similar Lie and Leibniz algebras.

\begin{thm} \label{thm71} Any superderivation on the Lie superalgebra $SL^{n,m}$ is inner.
\end{thm}
\begin{proof} Our goal is to prove that the following inner superderivations $\{ adx_1, \dots, adx_n,$ $ adt_1,adt_2,adt_3,$ $ady_1, \dots, $ $ady_m \}$ form a basis of the space of superderivations $Der(SL^{n,m})=Der _{\overline{0}}(SL^{n,m})\oplus Der _{\overline{1}}(SL^{n,m})$, with $\{ adx_1, \dots, adx_n, adt_1,adt_2,adt_3\}$ a basis of $Der _{\overline{0}}(SL^{n,m})$  and $\{ ady_1, \dots, ady_m\}$ a basis of $Der _{\overline{1}}(SL^{n,m})$.

Let $D$ be an  even superderivation of $SL^{n,m}.$ Then taking into account the embeddings  $D(SL^{n,m}_{\overline{0}})\subset SL^{n,m}_{\overline{0}}$ and $D(SL^{n,m}_{\overline{1}})\subset SL^{n,m}_{\overline{1}}$ we set
$$\begin{array}{lll}
D(x_1)=\displaystyle\sum_{s=1}^3\alpha_s t_s+\displaystyle\sum_{k=1}^n a_k x_k,& D(x_2)=\displaystyle\sum_{s=1}^3\beta_s t_s+\displaystyle\sum_{k=1}^n b_k x_k,& D(y_1)=\displaystyle\sum_{r=1}^m p_r y_r,\\[3mm]
D(t_1)=\displaystyle\sum_{s=1}^3\gamma_s t_s+\displaystyle\sum_{k=1}^n c_k x_k,&
D(t_2)=\displaystyle\sum_{s=1}^3\delta_s t_s+\displaystyle\sum_{k=1}^n d_k x_k,&
D(t_3)=\displaystyle\sum_{s=1}^3\nu_s t_s+\displaystyle\sum_{k=1}^n e_k x_k.
\end{array}
 $$

Acording to the even superderivation condition, we can summarized the computation in the following table:

\begin{center}
\begin{tabular}{c|c}
	Pairs& Constraints\\
	\hline \hline
$\{x_1,t_1\}$& $\alpha_i=0,\ 1\leq i\leq 3$, $\gamma_1=0,\ \alpha_2=0,$ $c_k=ka_{k+1}, \ 2\leq k\leq n-1$\\
\hline
$\{x_1,t_2\}$&$\delta_1=0,$ $d_k=a_{k+1},\ 2\leq k\leq n-1$\\
\hline
$\{x_1,t_3\}$& $\nu_1=0,$ $e_k=0,\ 2\leq k\leq n-1$\\
\hline
$\begin{array}{c}
\{x_1,y_{j-1}\}\\
2\leq j\leq m-1
\end{array}$& $d(y_j)=((j-1)a_1+p_1)y_j+\displaystyle\sum_{k=j+1}^m p_{k-j+1} y_k,\ 2\leq j\leq m$ \\
\hline
$\{y_1,x_2\}$&$b_1=0,\ \beta_3=-\beta_1$\\
\hline
$\{x_1,x_2\}$& $d(x_3)=-\beta_1 x_1+(a_1+b_2)x_3+\displaystyle\sum_{k=4}^n b_{k-1}x_k$\\
\hline
$\{x_3,y_1\}$&$\beta_1=0$ $\Rightarrow$ $\beta_3=0$\\
\hline
$\begin{array}{c}
\{x_1,x_{i-1}\}\\
4\leq j\leq n-1
\end{array}$& $d(x_i)=((i-2)a_1+b_2)x_i+\displaystyle\sum_{k=i+1}^n b_{k-i+2} x_k,\ 3\leq i\leq n$ \\
\hline
%
$\{t_3,t_1\}$&$e_1=e_n=0$\\
\hline
$\{t_3,x_2\}$&$\nu_2=0$\\
\hline
$\{t_1,x_2\}$&$\beta_2=\gamma_2=0,$ $c_1=-b_3,$ $b_k=0,\ 4\leq k\leq n$\\
\hline
$\{x_2,t_2\}$&$\delta_2=d_1=0$\\
\hline
$\{t_1,t_2\}$&$c_n=n d_n$\\
\hline
%
$\{t_1,y_1\}$&$\gamma_3=0,$ $p_2=b_3,$ $p_k=0,\ 3\leq k\leq m$\\
\hline
$\{t_2,y_1\}$&$\delta_3=0$\\
\hline
$\{t_3,y_1\}$&$\nu_3=0$\\
\hline

\end{tabular}
\end{center}

Therefore, we get
$$\begin{array}{lll}
D(t_1)=-b_3x_1+\displaystyle\sum_{k=2}^{n-1}ka_{k+1} x_k+nd_nx_n,&
D(t_2)=\displaystyle\sum_{k=2}^{n-1} a_{k+1} x_k+d_nx_n,&
D(t_3)=0,\\
D(x_1)=a_1x_1+\displaystyle\sum_{k=3}^n a_k x_k,& D(x_2)=b_2x_2+ b_3 x_3,&\\
D(x_i)=((i-2)a_1+b_2)x_i+b_3x_{i+1},& 3\leq i\leq n,&\\[2mm]
D(y_1)= p_1 y_1+b_3 y_2,& D(y_j)=((j-1)a_1+p_1)y_j+b_3 y_{j+1},& 2\leq j\leq m. \\[3mm]
\end{array}$$

Thus, we conclude $dim(Der _{\overline{0}}(SL^{n,m}))=n+3$. On the other hand, the $(n+3)$ inner  superderivations $\{ adx_1, \dots, adx_n, adt_1,adt_2,adt_3\}$ are in particular even superderivations. Hence, we obtain a basis of the space $Der _{\overline{0}}(SL^{n,m})$ composed by inner even superderivations. Moreover, $D$ can be expressed via inner as follows
$$D=b_3(adx_1)-\left(\sum_{k=2}^{n-2} a_{k+1}(adx_k)\right)-d_n(adx_n)+ a_1(adt_1-2adt_2-adt_3)+b_2(adt_2)+p_1(adt_3).$$

Analogously, we are going to compute the odd superderivations. Let $D$ now be an odd superderivation. We put
$$\begin{array}{lll}
D(x_1)=\displaystyle\sum_{k=1}^m a_k y_k,& D(x_2)=\displaystyle\sum_{k=1}^m b_k y_k,& D(y_1)=\displaystyle\sum_{s=1}^3p_s t_s+\displaystyle\sum_{r=1}^n c_r y_r,\\[3mm]
D(t_1)=\displaystyle\sum_{k=1}^m d_k y_k,&
D(t_2)=\displaystyle\sum_{k=1}^m g_k y_k,&
D(t_3)=\displaystyle\sum_{k=1}^m h_k x_k.
\end{array}$$

According to the odd superderivation conditions we have the following computations:
\begin{center}
	\begin{tabular}{c|c}
		Pairs& Constraints\\
		\hline \hline
	$\begin{array}{c}
	\{x_1,x_{i-1}\}\\
	3\leq j\leq n-1
	\end{array}$& $d(x_i)=\displaystyle\sum_{k=i-1}^m b_{k-i+2} y_k,\ 3\leq i\leq n$ \\
	\hline
	$\{x_1,y_1\}$&$d(y_2)=-p_1x_1+\displaystyle\sum_{k=3}^n c_{k-1}x_k$\\
	\hline
	$\begin{array}{c}
		\{x_1,y_{j-1}\}\\
		3\leq j\leq m-1
		\end{array}$& $d(y_j)=\displaystyle\sum_{k=j+1}^n c_{k-j+1} x_k,\ 2\leq j\leq m$ \\
		\hline
		$\{x_2,y_2\}$&$p_1=0$\\
		\hline
		$\{x_2,y_1\}$&$p_2=c_1=0$\\
		\hline
		$\{t_1,y_1\}$& $p_3=0,$ $c_k=0,\ 2\leq k\leq n$\\
		\hline
		$\{x_1,t_1\}$&$d_k=ka_{k+1},\ 1\leq k\leq m-1$\\
		\hline
		$\{x_1,t_2\}$&$g_k=0,\ 1\leq k\leq m-1$\\
		\hline
		$\{x_1,t_3\}$&$h_k=a_{k+1},\ 1\leq k\leq m-1$\\

		\hline
		$\{x_2,t_1\}$&$b_1=b_k=0,\ 3\leq k\leq m$\\

		\hline
		$\{x_2,t_2\}$&$b_2=0$\\
		
		\hline
		$\{t_1,t_2\}$&$g_m=0$\\
		
		\hline
		$\{t_1,t_3\}$&$d_m=mh_m$\\
		\hline
		
	\end{tabular}
\end{center}

Thus, we get
$$D(t_1)=\displaystyle\sum_{k=1}^{m-1}ka_{k+1} y_k+m h_m y_m, \quad D(t_2)=0, \quad D(t_3)=\displaystyle\sum_{k=1}^{m-1} a_{k+1} y_k+h_m y_m,$$
$$D(x_1)=\displaystyle\sum_{k=2}^m a_k y_k,\quad D(x_i)=0, \quad 2\leq i\leq n, \quad D(y_j)=0, \quad  1\leq j\leq m.$$

This imply that $dim(Der _{\overline{1}}(SL^{n,m}))=m$. On the other hand, we have $m$ odd inner superderivations (they are $\{ ady_1, \dots, $ $ady_m\}$). Consequently, a basis of $Der _{\overline{1}}(SL^{n,m})$ form by inner odd superderivations. In particular, $D$ can be expressed via inner superderivations as follows:

$$D=-\left(\sum_{k=1}^{m-1} a_{k+1}(ady_k)\right)-h_m(ady_m).$$
\end{proof}
Note that all the computations have been duplicated by using the software Mathematica.

Consider now the maximal-dimensional solvable Lie superalgebra with model nilpotent nilradical  $SN(n_1, \cdots,n_k,1 | m_1, \cdots, m_p).$

\begin{thm} \label{thm72} Any superderivation of the Lie superalgebra $SN(n_1, \cdots,n_k,1 | m_1, \cdots, m_p)$ is inner.
\end{thm}

\begin{proof} We are going to prove that the following inner superderivations $$\{ adx_1, \dots, adx_{n_1+\cdots n_k +1}, adt_1,\dots,adt_{k+1},adt'_1,\dots,adt'_p, ady_1, \dots, ady_{m_1+\cdots+m_p}\}$$ form a basis of the space $Der(SN(n_1, \cdots,n_k,1 | m_1, \cdots, m_p))$ with $\{ adx_1, \dots, adx_{n_1+\cdots n_k +1},$ $adt_1,\dots,$ $adt_{k+1},$ $adt'_1,\dots,adt'_p\}$ a basis for even superderivations, $Der _{\overline{0}}(SN(n_1, \cdots,n_k,1 | m_1, \cdots, m_p))$,  and $\{ ady_1, \dots, $ $ady_{m_1+\cdots+m_p}\}$ a basis for the odd ones, $Der _{\overline{1}}(SN(n_1, \cdots,n_k,1 | m_1, \cdots, m_p))$.

Let $D$ be an even superderivation of $SN(n_1, \cdots,n_k,1 | m_1, \cdots, m_p)$. Then from superderivation property we derive 
	$$\begin{array}{l}
	D(x_1)=\alpha_1x_1+\displaystyle\sum_{s=3}^{n_1+1}  \alpha_{s}x_s
	+\displaystyle\sum_{j=1}^{k-1}\left(\sum_{s=n_1+\dots+n_j+3}^{n_1+\dots+n_{j+1}}\alpha_{s} x_s\right),\\[2mm]
	D(x_i)=((i-2)\alpha_2+\beta)x_i,\quad 2\leq i\leq n_1+1,\\[2mm]
		
	D(x_{n_1+\dots+n_j+i})=((i-2)\alpha_1+a_j) x_{n_1+\dots+n_j+i},\quad  1\leq j\leq k-1,\ 2\leq i\leq n_{j+1}+1,\\[3mm]
	
	D(t_1)=\displaystyle\sum_{s=2}^{n_1} s \alpha_{s+1}x_s
	+\displaystyle\sum_{j=1}^{k-1}\left(\sum_{s=n_1+\dots+n_j+2}^{n_1+\dots+n_{j+1}}s\alpha_{s+1} x_s\right)+\\[3mm]
\qquad \
	+\displaystyle\sum_{j=1}^{k}(n_1+\dots+n_j+1) b_j x_{n_1+\dots+n_j+1},\\[3mm]
D(t_2)=\displaystyle\sum_{s=2}^{n_1}  \alpha_{s+1}x_s
+ b_1 x_{n_1+1},\\[3mm]
D(t_{j+2})=\displaystyle\sum_{s=n_1+\dots+n_j+2}^{n_1+\dots+n_{j+1}}  \alpha_{s+1}x_s
+ b_{j+1} x_{n_1+\dots+n_{j+1}+1},\qquad 1\leq j\leq k-1,\\[2mm]
	
D(t'_j)=0,\qquad 1\leq j\leq p,\\[2mm]
D(y_i)=((i-1)\alpha_1+\gamma)y_i,\qquad 1\leq i\leq m_1,\\[3mm]
D(y_{m_1+\dots+m_j+i})=((i-1)\alpha_1+q_j)  y_{m_1+\dots+m_j+i},\quad 1\leq j\leq p-1,\ 1\leq i\leq m_{j+1}.

	\end{array}$$

	Let $D$ be an odd superderivation of $SN(n_1, \cdots,n_k,1 | m_1, \cdots, m_p)$. Then the superderivation property imply 
$$\begin{array}{l}
D(x_1)=\displaystyle\sum_{s=2}^{m_1}  \alpha_{s}y_s
+\displaystyle\sum_{j=1}^{p-1}\left(\sum_{s=m_1+\dots+m_j+2}^{m_1+\dots+m_{j+1}}\alpha_{s} y_s\right),\\[3mm]
D(x_i)=0,\quad 2\leq i\leq n_1+\dots+n_k+1,\\[3mm]
D(t_1)=\displaystyle\sum_{s=1}^{m_1-1} s \alpha_{s+1}y_s
+\displaystyle\sum_{j=1}^{p-1}\left(\sum_{s=m_1+\dots+m_j+1}^{m_1+\dots+m_{j+1}-1}s\alpha_{s+1} y_s\right)+\displaystyle\sum_{j=1}^{p} (m_1+\dots+m_j)\delta_j y_{m_1+\dots+m_j},\\[3mm]
D(t_i)=0,\quad 2\leq i\leq k+1,\\[3mm]
D(t'_1)=\displaystyle\sum_{s=1}^{m_1-1} \alpha_{s+1}y_s
+\displaystyle\sum_{j=1}^{p} \delta_j y_{m_1+\dots+m_j},\\[2mm]
D(t'_i)=0,\quad 2\leq i\leq p,\\[2mm]
D(y_j)=0,\quad 1\leq j\leq m_1+\dots+m_p.
\end{array}$$

Therefore, we have
$$dim(Der _{\overline{0}}(SN(n_1, \cdots,n_k,1 | m_1, \cdots, m_p))=(n_1+\cdots n_k +1)+k+1+p,$$
$$dim(Der _{\overline{1}}(SN(n_1, \cdots,n_k,1 | m_1, \cdots, m_p))=m_1+\cdots+m_p.$$

Now, since both sets of inner superderivations of the statement of the Theorem, even and odd, are linearly independent we conclude that the set $$\{ adx_1, \dots, adx_{n_1+\cdots n_k +1}, adt_1,\dots,adt_{k+1},adt'_1,\dots,adt'_p, ady_1, \dots, ady_{m_1+\cdots+m_p}\}$$ constitutes a basis of the superalgebra of superderivations $Der(SN(n_1, \cdots,n_k,1 | m_1, \cdots, m_p))$.
\end{proof}

Since the proofs of the following results based on the application the same arguments and similar computations as in the proofs of Theorems \ref{thm71} and \ref{thm72} we present summaries of their proofs.

\begin{thm} Any superderivation of the Leibniz superalgebra $SLP^{n,m}$ is inner.
\end{thm}
\begin{proof} As a result of computing of the odd and even superderivations properties of the Leibniz superalgebra $SLP^{n,m}$ we obtain $Der_{\bar{1}}(SLP^{n,m})=\{0\}$ and for an arbitrary $d\in Der_{\bar{0}}(SLP^{n,m})$ we get $$\begin{array}{lll}
d(x_1)=\alpha x_1,&
d(x_i)=(\beta+(i-2)\alpha) x_i-\gamma x_{i+1},&2\leq i\leq n,\\{}
d(t_1)=\gamma x_1,&
d(t_2)=d(t_3)=0,  &\\{}
d(y_1)=\delta y_1-\gamma y_2,&
d(y_j)=(\delta+(j-1)\alpha) y_j-\gamma y_{j+1},&2\leq j\leq m,
\end{array}$$
for some parameters $\alpha, \beta, \gamma, \delta.$

Consequently, $dim(Der(SLP^{n,m}))=4$ and hence, we obtain $Der(SLP^{n,m})=span\{R_{x_1},$ $R_{t_1},$ $R_{t_2},$ $R_{t_3}\}$.
In particular, $d$ can be expressed via inner derivations as follows:
$$d=-\gamma R_{x_1}+\alpha R_{t_1}+\beta R_{t_2}+\delta R_{t_3}.$$
\end{proof}

\begin{thm} Any superderivation of the Leibniz superalgebra $SNP(n_1, \cdots,n_k,1 | m_1,\cdots, m_p)$ is inner.
\end{thm}
\begin{proof}
	Analogously to the previous superalgebra, we obtain that
	$Der_{\bar{1}}(SLP(n_1, \cdots,n_k,1 | m_1,\cdots, m_p))=\{0\}$ and for an arbitrary even superderivation $d$ we derive the following:
	$$\begin{array}{ll}
	d(x_1)=\alpha x_1,&\\{}
	d(x_i)=(\beta+(i-2)\alpha)x_i-\gamma x_{i+1},&3\leq i\leq n_1,\\{}
	d(x_{n_1+1})=(\beta+(n_1-1)\alpha)x_{n_1+1},&\\{}
	d(x_{n_1+\dots+n_j+2})=\mu_j x_{n_1+\dots+n_j+2}-\gamma x_{n_1+\dots+n_j+3},&1\leq j\leq k-1,\\{}
	d(x_{n_1+\dots+n_j+i})=(\mu_j+(i-2)\alpha) x_{n_1+\dots+n_j+i}-\gamma x_{n_1+\dots+n_j+i+1},&1\leq j\leq k-1,\ 3\leq j\leq n_{j+1},\\{}
	d(x_{n_1+\dots+n_{j+1}+1})=(\mu_j+(n_{j+1}-1)\alpha) x_{n_1+\dots+n_{j+1}+1},&1\leq j\leq k-1,\\{}
	
	d(t_1)=\gamma x_1,&\\{}
	d(t_i)=d(t'_j)=0,&2\leq i\leq k+1,\ 1\leq j\leq p,\\{}
d(y_1)=\delta y_1-\gamma y_2,&\\{}
d(y_j)=(\delta +(j-1)\alpha) y_j-\gamma y_{j+1},& 2\leq j\leq m_1-1,\\{}
d(y_{m_1})=(\delta+(m_1-1)\alpha)y_{m_1},&\\{}

d(y_{m_1+\dots+m_s+1})=\nu_s y_{m_1+\dots+m_s+1}-\gamma y_{m_1+\dots+m_s+2},& 1\leq s\leq p-1,\\{}
d(y_{m_1+\dots+m_s+i})=(\nu_s+(i-1)\alpha) y_{m_1+\dots+m_s+i}-\gamma y_{m_1+\dots+m_s+i+1},& 1\leq s\leq p-1,\ 2\leq i\leq m_{s+1}-1,\\{}
d(y_{m_1+\dots+m_{s+1}})=(\nu_s+(m_{s+1}-1)\alpha) y_{m_1+\dots+m_{s+1}},& 1\leq s\leq p-1.
	\end{array}$$
	
	Then, $dim(Der(SLP(n_1, \cdots,n_k,1 | m_1,\cdots, m_p)))=k+p+2.$ On the other hand, we have $k+p+2$ inner derivations, $\{R_{x_1},R_{t_1},R_{t_2},R_{t_3},\dots,R_{t_{k+1}},R_{t'_1},R_{t'_{2}},\dots,R_{t'_{p}}\}.$ 
 In particular, $d$ can be expressed via inner derivations as follows:
	
$$d=-\gamma R_{x_1}+\alpha R_{t_1}+\beta R_{t_2}+\sum_{j=1}^{k-1}\mu_j R_{t_{j+2}}+\delta R_{t'_1}+\sum_{s=1}^{p-1}\nu_s R_{t'_{s+1}}.$$

\end{proof}

\bibliographystyle{amsplain}

\end{document}